\documentclass[12pt]{article}
\usepackage[english]{babel}
\usepackage[latin1]{inputenc}
\usepackage{amsmath,amssymb,amsthm,xspace,enumerate,url}
\renewcommand{\P}{{\bf \mathcal{P}}}
\renewcommand{\L}{{\bf \mathcal{L}}}

\newcommand{\Ind}{
 \setbox0=\hbox{$x$}\kern\wd0\hbox to 0pt{\hss$
 \mid$\hss}\lower.9\ht0\hbox to 0pt{\hss$\smile$\hss}\kern\wd0
}
\newcommand{\indep}[3]{
 #1\mathop{\mathpalette\Ind{}}_{#2}#3
}

\newcommand{\dInd}{
 \setbox0=\hbox{$x$}\kern\wd0\hbox to 0pt{\hss$
 \mid^d$\hss}\lower.9\ht0\hbox to 0pt{\hss$\smile$\hss}\kern\wd0
}
\newcommand{\dnInd}{
 \setbox0=\hbox{$x$}\kern\wd0\hbox to 0pt{\hss$
 \mid^d$\hss}\lower.9\ht0\hbox to 0pt{\hss$\smile$\hss}\kern\wd0
}
\newcommand{\dindep}[3]{
 #1\mathop{\mathpalette\dInd{}}_{#2}#3
}

\newcommand{\Notind}{
 \setbox0=\hbox{$x$}\kern\wd0\hbox to 0pt{\mathchardef
 \nn=12854\hss$\nn$\kern1.4\wd0\hss}\hbox to 0pt{\hss$\mid$\hss}\lower.9\ht0
 \hbox to 0pt{\hss$\smile$\hss}\kern\wd0
}
\newcommand{\Notdind}{
 \setbox0=\hbox{$x$}\kern\wd0\hbox to 0pt{\mathchardef
 \nn=12854\hss$\nn$\kern1.4\wd0\hss}\hbox to 0pt{\hss$\mid^d$\hss}\lower.9\ht0
 \hbox to 0pt{\hss$\smile$\hss}\kern\wd0
}

\newcommand{\notind}[3]{
 #1\mathop{\mathpalette\Notind{}}_{#2}#3
}

\title{Building-like geometries of finite Morley Rank\\
}

\author{Isabel M\"uller\footnote{Partially supported by SFB 878 and \emph{Deutsche Telekom Stiftung}}\quad and Katrin Tent
}
\date{May 29, 2017}

\newtheorem{satz}{Theorem}[section]
\newtheorem{theorem}[satz]{Theorem}

\newtheorem{lemma}[satz]{Lemma}
\newtheorem{proposition}[satz]{Proposition}
\newtheorem{corollary}[satz]{Corollary}
\newtheorem{definition}[satz]{Definition}
\newtheorem{remark}[satz]{Remark}

\newcommand{\nc}{\newcommand}
\nc{\sa}{semialgebraic\xspace}
\nc{\el}{elementary\xspace}
\nc{\low}{lower \el}
\nc{\inv}[1]{\frac{1}{#1}}
\nc{\G}{\Gamma}
\nc{\Np}{\N_{\scriptscriptstyle >0}}
\nc{\Z}{\mathbb{Z}}
\nc{\Q}{\mathbb{Q}}
\nc{\N}{\mathbb{N}}

\nc{\Rp}{\R_{\scriptscriptstyle >0}}
\nc{\C}{\mathbb{C}}
\nc{\F}{\ensuremath{\mathcal{F}}\xspace}
\nc{\K}{\mathcal{K}}
\nc{\Kmu}{\mathcal{K_\mu}}

\nc{\Mmu}{M_\mu}
\nc{\Tmu}{T_\mu}

\nc{\grad}{\chi^}
\nc{\U}{\mathbb{U}}

\nc{\E}{\mathbb{E}}
\nc{\Epsilon}{{\Large $\epsilon$}} 
\nc{\ap}{approximable\xspace}
\nc{\e}{\mathrm{e}}
\nc{\ii}{\,\mathrm{i}}
\nc{\Es}{\E\setminus\R_{\scriptscriptstyle\leq0}}

\DeclareMathOperator{\MR}{MR}

\DeclareMathOperator{\cl}{cl}
\DeclareMathOperator{\acl}{acl}
\DeclareMathOperator{\dcl}{dcl}
\DeclareMathOperator{\res}{res}

\begin{document}

\maketitle
\begin{abstract}
For any $n\geq 6$ we construct almost strongly minimal geometries of type $\bullet \overset{n}{-} \bullet \overset{n}{-}\bullet$ which are $2$-ample but not $3$-ample.\footnote{Parts of the results in this paper are also contained in M\"uller's PhD thesis.}
\end{abstract}

\section{Introduction}

In the investigation of geometries on strongly minimal sets the notion of
ampleness plays an important role, see \cite{Hrushovski}, \cite{HrushovskiZilber}. The notion reflects the geometry of
a projective space and hence projective spaces and more generally Tits buildings (of dimension $n+1$) are the canonical examples of $n$-ample structures. By the fundamental theorem of projective geometry, projective spaces of dimension at least $3$ arise from fields and these fields can be recovered from the geometry.
The same holds for spherical Tits buildings of rank at least $3$. 
Since algebraically closed fields are $n$-ample for all $n$ there were no known geometries of finite Morley rank which are $2$-ample, but do not interpret an infinite field, and it has been a long-standing open question whether such strongly minimal sets exist.

We here  construct strongly minimal sets arising from geometries of
\emph{geometric} rank $3$ which are $2$-ample, but not $3$-ample and hence do not interpret any infinite field. By the geometric rank of a geometry we mean the number of different sorts of vertices in the associated graph, which in this case we think of as points, lines and planes. In the geometries constructed here the residues of points and planes consist of generalized polygons constructed in \cite{Te1}.
Since generalized polygons are a generalization of projective planes and these
appear as the residues in projective spaces of higher rank,
 the geometries constructed here are the 'obvious' higher rank analogues to the higher rank buildings containing the  generalized polygons from \cite{Te1} as their residues. We expect that this construction can be
extended to yield geometries of rank $n+1$ which are $n$-ample, but not $(n+1)$-ample.

Note that the right-angled Tits buildings investigated in \cite{Te2} and \cite{BPZ} are $\omega$-stable of infinite Morley rank and for a building to have finite Morley rank, it has to be of spherical type. Since the results from \cite{KTVM} show that in 
spherical buildings of finite Morley rank the field is interpretable, the geometries constructed here cannot be buildings.

The construction uses a $\delta$-function closely related to the function from \cite{Te1}. However, in contrast to all other known Hrushovski constructions of
structures of finite Morley rank, in these examples the associated predimension function $\delta$ is not submodular.

\section{Construction of the geometry}

>From now on we fix some $n\geq 6$. 
We construct a geometry $\Gamma$ consisting of points, lines and planes with incidence
given by edges in the corresponding graph. As in \cite{Te1} we use Hrushovski amalgamation using a predimension function $\delta$ for the construction.

The geometry  will be realized as a  $3$-coloured graph whose vertices we may think of as points, lines, 
and planes. 
The edge relation describes the incidence between points, lines and planes. We also say that a point
(or line) is contained in a line or plane (respectively) if it is incident with it.
A \emph{flag} is a tuple of pairwise incident elements of distinct sorts and a flag is \emph{complete} if it consists of a point, line and plane. Note here that no element is incident to itself.

We work in a relational language $L_3$ containing predicates $\P, \L,\Pi$ so that a vertex of the graph is considered as a point, line or plane if it belongs to $\P,\L,$ or $\Pi$, respectively. We also add two kinds of edge relations: we let $E\subset (\P\times\L)\cup (\L\times\Pi)$ denote edges between points and lines or lines and planes, respectively. Edges between points and planes will be denoted by $E^2\subset \P\times\Pi$. (Note that $E^2$ is a binary relation, and not $E\circ E$.) Finally we add a predicate $\F$ for triples forming a complete flag. 

For any $x\in \G$ the residue $\res(x)$ denotes the set of vertices incident with~$x$. Thus if $x\in\P$, then $\res(x)\subset\L\cup\Pi$ etc.

Recall that a generalized $n$-gon is a bipartite graph
of diameter $n$ not containing any cycles of length less than $2n$.

\begin{definition}
We call a $3$-coloured graph $\G=\G_n$  a geometry of type $\bullet \overset{n}{-} \bullet \overset{n}{-}\bullet$
if the following conditions are satisfied:
\begin{enumerate}
\item For any $x\in \P\cup\Pi$, its residue $\res(x)$ is a generalized $n$-gon;
\item  a point is incident with
a plane if and only if there is a line incident with both, so
for any $x\in\L$, $\res(x)$ is a complete bipartite graph\footnote{Note that this implies that in a model $E^2=E\circ E$.};
\item for any two points (two planes, respectively) there is at most one line containing both; 
\item for any two points (two planes, respectively) not contained in a common line there is a unique plane (point, respectively) containing both.
\end{enumerate}
\end{definition}

We remark that the definition of a geometry of type $\bullet \overset{n}{-} \bullet \overset{n}{-}\bullet$ is \emph{self-dual}, i.e. symmetric in points and planes.
We call the \emph{dual} of a statement $\varphi$  the statement obtained from $\varphi$ by switching the roles of points and planes. 

Note that  the geometries of type  $\bullet \overset{n}{-} \bullet \overset{n}{-}\bullet$ form an
elementary class in the language $L_3$.
All graphs we consider will be $3$-coloured, so we may
omit mentioning it.
For a finite $3$-coloured graph $A$ we now put
 \begin{equation*} 
\delta(A):=(3(n-1)-1)|\L_A|+2(n-1)|\P_A\cup\Pi_A|
-(2(n-1)-1)|E_A|-(n-1)(|E^2_A|-|F_A|)
 \end{equation*}
where $\Pi_A, \L_A,\P_A, E_A, E^2_A$ and $F_A$ denotes the set of planes, lines, etc. in the graph $A$.
As usual we also define for finite $A$ and $ B$ (contained in some common graph $C$)
\[\delta(B/A):=\delta(AB)-\delta(A)\]
where we write $AB$ for the induced subgraph on the union $A\cup B$.
We say that a  finite graph $A$ is \emph{strong} in a graph $B$ (and we write $A\leq B$) if $A$ is contained in $B$ and for any finite subgraph $C$ of $B$ we have $\delta(C/A)\geq 0$.

For $A\subset \P\cup\L$ or $A\subset \L\cup\Pi$ we also define
\[\delta_1(A)=(n-1)|A|-(n-2)|E_A|.\]

This is the delta function from \cite{Te1} for generalized $n$-gons. Since generalized $3$-gons are nothing but projective planes, we see in particular, that for $n=3$ this is the delta function for a projective plane:
\[\delta_{proj}(A)=2|A|-|E_A|.\]

Note that we have the following, which motivates the choice of the
delta function given above:

\begin{lemma}\label{l:polygon}
For $x\in \P\cup\Pi$ and $A\subseteq\res(x)$  we have
\[\delta(A/x)=\delta_1(A).\]
\end{lemma}
\noindent
Before giving the (easy) proof we introduce the following notation:

\noindent
{\bf Notation}: For sets $A,B$ we let $E(A,B), E^2(A,B)$ and $F(A,B)$ denote the
set of edges or flags, respectively, containing one vertex from $A$ and one from $B$
and in case of flags the third one from $AB$.
If we need to be more specific, we may also write $F(A,B,C)$ to denote the set of flags having exactly one vertex from each of the sets $A,B,C$ (but without fixed order).

\begin{proof} (of Lemma~\ref{l:polygon})
By symmetry we may assume $x\in\Pi$. We then have \[\delta(A/x)=\delta(A)-(2(n-1)-1)|E(A,x)|-(n-1)|E^2(A,x)|+(n-1)|F(A,x)|\]
\[=\delta_1(A)+(2n-3)\left(|\L_A|-|E(A,x)|\right)+(n-1)\left(|\P_A|-|E^2(A,x)|\right)+(n-1)\left(|F(A,x)|-|E_A|\right).\]
Clearly,  the lines in $A$ correspond to the edges in $E(A,x)$, the points  in $A$ correspond to the edges in $E^2(A,x)$ and the edges of $A$ correspond to flags containing $x$. This proves the claim.
\end{proof}

As in \cite{Te1} the previous lemma immediately implies:
\begin{remark}\label{r:pathextension}
For any plane $x$ and vertices $a,b\in\res(x)$ an extension of $A=\{a,b,x\}$ by a path $\gamma=(x_0=a,x_1,\ldots,x_k=b)\subset\res(x)$ is strong if and only if $k\geq n-1$.
\end{remark}
\begin{remark}\label{r:projplane}
Note that restricted to points and planes, the delta function
reads as a multiple of $\delta_1=\delta_{proj}$ in the case of a projective plane.
Thus, if $A$ contains only points and planes, then we have
\[\delta(A)=2(n-1)|A|-(n-1)|E^2_A|=(n-1)\delta_{proj}(A).\]
\end{remark}

\begin{definition}\label{d:K}
Let $\K$ denote the class of $3$-coloured graphs $A$ satisfying the following conditions:
\begin{enumerate}
\item for any two points in $A$ there is at most one line in $A$ containing them, similarly for planes;
\item if two points are contained in a common line, this line is contained in any plane containing the two points; similarly for the dual;
\item if two points of $A$ are not contained in a common line, then there is at most one plane in $A$ containing them, similarly for planes;
\item if a point $p\in A$ is incident with a line $l\in A$ and $l$ is incident with a plane $e\in A$, then $p$ is incident with $e$;
\item for any point or plane $x\in A$, there are no cycles of length less than $2n$ in $\res(x)$;
\item for any finite subgraph $B\subseteq A$, $B\in\K$ we have \begin{enumerate}
\item if $|B|\geq 3$, then
$\delta(B)\geq 3(n-1)+1$;
\item if $B$ contains a point $x$ such that $B\cap\res(x)$ contains a $2k$-cycle for $k\geq n+1$, we have $\delta_1(B\cap\res(x))\geq 2(n+1)$.
\end{enumerate} 
\end{enumerate}
\end{definition}

In view of condition 4. we call an edge $(x,y)$ of type $E^2$ an \emph{induced} edge if there is a line $w$ such that
$(x,w)$ and $(w,y)$ are edges.

\begin{remark}\label{r:substructures}
Clearly, $\K$ is an elementary class in the relational language $L_3$ introduced above.
Note however that by property 3. the class $\K$ is not closed under substructures when
considered in this language. However, $\K$ is closed under intersections.
\end{remark}

The definition of $\K$ immediately implies the following:

\begin{remark}\label{l:delta pos}\begin{enumerate}
\item For any $A\in \K$, $A\neq\emptyset$, we have $\delta(A)\geq 2(n-1)$. If $A$ contains a line, then $\delta(A)\geq 3(n-1)-1$. If $ A$ contains a complete flag, then $\delta(A)\geq 3(n-1)+1$;

\item For any $x\in\P\cup\Pi$ and $A\cup\{x\}\in \K,A\neq\emptyset, A\subseteq \res(x)$, we have $\delta(A/x)=\delta_1(A)\geq n-1$.
\end{enumerate}
\end{remark}

Note that  this delta function is \emph{not submodular}, i.e. there are sets $A\subset B$ and $C$ such that
\[ \delta(C/A)<\delta(C/B).\]

Namely, if $A=\{p_1,p_2\}$ consists of two points, $B=A\cup \{l\}$ for some line $l$ incident to the points in $A$ and $C=\{e\}$ for some plane $e$ containing $l$, then
we have
\[1=\delta(e/B)>\delta(e/A)=0.\]

This example motivates the following definition:
\begin{definition}\label{d:kstrong}
Suppose that $A\subset B$ are finite graphs. \begin{enumerate}
\item  We say that $A$ is $\L$-strong in $B$
and write $A\leq_\L B$ if
for all $x_1,x_2\in \P_A$ ($x_1,x_2\in\Pi_A$) any $l\in\L_B$ with $E(l,x_i),i=1,2$ is contained in $A$.

\item More generally we define $A\leq_kB$ for $A,B\in\K$ if $A$ is strong in every $k$-element extension of $A$ in $B$, i.e. if for every $B_0\subset B$ 
with $|B_0\setminus A|\leq k$ we have $A\leq AB_0$. Note that $A\leq_1 B$ implies $A\leq_\L B$.

\item For a proper strong extension $A\leq B$ we say that $B$ is minimal over $A$ if there is no  proper subset $C\neq A$ of $ B$ such that $A\leq C$ and $C\leq B$. 
\end{enumerate}
\end{definition}

We will show that in strong extensions the previous example is  the only source of failing submodularity. For this we first show

\begin{lemma}\label{l:delta_1}
If $\delta(x/A)\geq 0$ for $x\in \P\cup\Pi$, then there is at most one line in $\res(x)\cap A$. If $l\in\res(x)\cap A$ is a line, then
all other vertices of $\res(x)\cap A$ are in $\res(l)\cap\res(x)$.
If there is no such line, then $|A\cap\res(x)|\leq 2$.
\end{lemma}

\begin{proof}
Assume that we have $x\in\Pi$ and a finite set $A\subset\res(x)$ with $A$ and $A\cup\{x\}\in\K$ and
\[\delta(x/A)=  2(n-1)-(2(n-1)-1)|\L_A|-(n-1)|\P_A|+(n-1)|E_A|\geq 0.\]
Then
\[|E_A|\geq \frac{2n-3}{n-1}|\L_A|+|\P_A|-2.\]

If $A$ contains no cycles, we have  $|E_A|<|\L_A|+|\P_A|$. It follows that $|\L_A|\leq 1$. If $|L_A|=1$,
then $|E_A|=|P_A|$ and if $A$ contains no line, then
$|\P_A|\leq 2$, proving the claim.

Now assume towards a contradiction that $A$ contains a cycle. As the previous inequality will be preserved when removing points of degree one, we may\ assume that all points in $A$ have degree at least $2$.
>From the above we have
\[n-1\geq (2(n-1)-1)|\L_A|+(n-1)|\P_ A|-(n-1)|E_ A|-(n-1).\]
Furthermore
\[\delta(A/x)=\delta_1(A)=(n-1)(|\L_A|+|\P_A|)-(n-2)|E_A|\geq n-1.\hspace{.5cm} (*)\]

Multiplying the last inequality by $(n-1)$ and the previous one by $(n-2)$, we get

\[(n-1)^2 (|\L_A|+|\P_ A|) \geq (n-2)(2n-3)|\L_A |+(n-2)(n-1)|\P_A |-(n-1)(n-2).\]
Hence \[(n-1)|\P_ A | \geq (n^2-5n+5)|\L_A|-(n^2-3n+2).\]

Since $|E_A| \geq 2|\P_A |$, we have from $(*)$
\[(n-1)(|\L_A |+|\P_A |) \geq (n-2)|E_A |+(n-1)\]
\[\geq 2(n-2)|\P_A |+(n-1).\]
Hence
\[(n-1)|\L_A | \geq (n-3)|\P_A |+(n-1).\]
Putting the above pieces together, we get
\[(n-1)^2 |\L_A| \geq (n-3)(n-1)|\P_A |+(n-1)^2\]
\[\geq (n-3)(n^2-5n+5)|L_A |-(n-3)(n^2-3n + 2)+(n-1)^2 ,\]
For $n \geq 6 $ this yields
\[5>\frac{n^3-7n^2+13n-7}{n^3-9n^2+22n-16}
\geq |\L_A |.\]
For $n \geq 6$, we must have $|\L_A| \leq 4$ contradicting our assumption that $A$ contains a cycle.

\end{proof}

We also need the corresponding result for lines:
\begin{lemma}\label{l:line}
Suppose $A$ and $Ax$ are in $ \K$ with $x\in\L$ and $\delta(x/A)\geq 0$. Then $A\cap\res(x)$ contains at most one point and one plane.
\end{lemma}
\begin{proof}
Put $A_0=A\cap res(x)$ and
suppose \[\delta(x/A)=(3(n-1)-1)-(2(n-1)-1)(|\P_{A_0}|+|\Pi_{A_0}|)+(n-1)|F(A,x,A)|\geq 0.\]
Since $A$ and $Ax$ are in $\K$, at least one of $\P_{A_0}$ and $\Pi_{A_0}$ contains at most one
element. Thus the number of flags in $F(A,x)$ is equal to the maximum of $|\P_{A_0}|$ and $|\Pi_{A_0}|$ if both numbers are non-zero. If this maximum is greater than one, we have a contradiction.
\end{proof}

The previous Lemmas~\ref{l:delta_1} and \ref{l:line} imply the following inductive setting for strong extensions:

\begin{lemma}\label{l:strongresidue}
Suppose that $A\leq C$ and $b\in\P_A\cup\Pi_A$.  Putting $A_0=A\cap\res(b)$ and  $C_0=C\cap\res(b)$  we have \[0\leq\delta(C_0/A)\leq \delta(C_0/A_0b)=\delta_1(C_0/A_0).\]In particular $A_0b\leq C_0b$. 
\end{lemma}
\begin{proof}
By symmetry we may assume that $b\in\Pi_A$. We may also assume $C=C_0\subseteq \res(b)$. Put  $\hat{A}=A\setminus A_0b$. 
Then 
\[\delta(C/A)=\delta(C/A_0b)-(2n-3)|E(C,\hat A)|-(n-1)|E^2(C,\hat A)|+(n-1)|F(C,\hat A,AC)|.\]

Note that by Lemma~\ref{l:line},  a line $c\in \L_C\cap\res(b)$ cannot be incident to any element of $\hat A$, so there are no flags in $F(c,\hat{A},AC)$. 
For a point $c\in \P_C\cap\res(b)$, by Lemma~\ref{l:delta_1} there is at most one line $a\in A$ incident to $c$ and in this case we have $a\in A_0$.  If there is no such line, there are no flags in $F(c,\hat{A},AC)$. Thus the only flags in $F(c,\hat A, AC)$ are of the form $(c,a,d)$ with $d\in\hat{A}$ and therefore $(n-1)|E^2(c,\hat{A})|\geq (n-1)|F(c,\hat{A}, AC)|$, proving the claim.
\end{proof}

\begin{lemma}\label{l:submodular}\label{c:submodular}
Suppose
\begin{enumerate}
\item  $AC\leq_\L BC$,
\item $B\cap C=\emptyset$, and 
\item $C$ is strong over $B$.
\end{enumerate}
Then $\delta(C/A)\geq\delta(C/B)$. 
\end{lemma}
\begin{proof}
Inductively it suffices to prove the lemma for $B=A\cup\{b\}$ since we may first remove points and planes and finally
the lines, one at a time, from $B\setminus A$. 

So suppose $\delta(C/A)<\delta(C/Ab)$ for some $C$ strong over $B=A\cup\{b\}$ and $AC\leq_\L BC$.
Then 
\[\delta(C/Ab)=\delta(C)-(2(n-1)-1)|E(C,Ab)|-(n-1)E^2(C,Ab)+(n-1)F(C,Ab)\]
\[>\delta(C/A)=\delta(C)-(2(n-1)-1)|E(C,A)|-(n-1)E^2(C,A)+(n-1)F(C,A)\]
implying
\begin{equation}
(n-1)F(C,b,AbC)>(2(n-1)-1)|E(C,b)|+(n-1)E^2(C,b).
\end{equation}

This inequality together with Lemma~\ref{l:strongresidue}
 shows that we may assume $A,C\subseteq\res(b)$.

First suppose that $b$ is a line. In this case we may rewrite the previous inequality as
\[
(n-1)F(C,b,AC)>(2(n-1)-1)|E(C,b)|.
\]
However, since $AC\leq_\L BC$ there can be at most one flag in $F(C,b,AC)$, this case cannot occur.

Hence up to  duality we may assume that $b$ is a plane. In this case we
rewrite the inequality (1) as
\begin{equation}
(n-1)|E(C,AC)|>(2n-3)|\L_C|+(n-1)|\P_C|.
\end{equation}

This inequality shows that we may assume that every point $x\in C$ has valency at least two in $CA$ and hence 
\[2|\P_C|\leq|E(C,AC)|.\]

Since  $Ab\leq AbC$ we have \[\delta(C/Ab)=\delta_1(C/A)=\delta_1(C)-(n-2)|E(C,A)|\geq 0\]
and thus
\[|E(C,A)|\leq \frac{1}{n-2}\cdot\delta_1(C)=\frac{n-1}{n-2}(|\P_C|+|\L_C|)-|E_C|.\]

Therefore we have \[2|\P_C|\leq |E(C,CA)|=|E(C,A)|+|E_C|\leq \frac{n-1}{n-2}(|\P_C|+|\L_C|).\]
Thus we conclude that
\[|\P_C|\leq \frac{n-1}{n-3}|\L_C|.\]

On the other hand, inequality (2) yields
\[\frac{(n-1)^2}{n-2}(|\P_C|+|\L_C|)\geq (n-1)|E(C,CA)|>(2n-3)|\L_C|+(n-1)|\P_C|\]
and hence
\[|\P_C|>\frac{n^2-5n+5}{n-1}|\L_C|.\]

Thus we have \[\frac{n-1}{n-3}|\L_C|\geq |\P_C|>\frac{n^2-5n+5}{n-1}|\L_C|\]
contradicting $n\geq 6$.
\end{proof}


\begin{definition}
For a proper strong extension $A\leq B$ with $\delta(B/A)=i$ we say that $B$ is  $i$-minimal over $A$ if there is no  proper subset $C\neq A$ of $ B$ such that $A\leq C$ and $C\leq B$. Thus $B$ is minimal over $A$ if $B$ is $i$-minimal over $A$ for some $i\geq 0$.
Furthermore, for disjoint sets $A, B$ we call $(A,B)$ an $i$-minimal pair if $A\leq AB$ is an $i$-minimal strong extension and every element of $A$ has a non-induced edge to an element of $B$. A $0$-minimal pair $(A,B)$ is called \emph{simple}. 
\end{definition}

We next fix a function $\mu$ from simple pairs $(A, B)$ with $AB\in \K$  into the natural numbers
satisfying the following properties:
\begin{enumerate}
\item $\mu(A, B)$ depends only on the isomorphism type of $(A, B)$;
\item $\mu(A,B)=1$ if $A$ consists of two points (planes, respectively) and $B$ consists of a single plane (point, respectively) incident with the elements of $A$;
\item $\mu(A,B)=1$ if $A=\{x,y,z\}$ with $x\in\P\cup\Pi$ and $y,z\in\res(x)$ and $B\subset\res(x)$ consists of a path from $y$ to $z$ containing exactly $n-2$ new elements;
\item $\mu(A,B)\geq 2\delta(A)$ except if $(A,B)$ is as in 2. or 3.

\end{enumerate}

Note that  $A\neq \emptyset$ by Remark~\ref{l:delta pos} and hence $\mu(A, B)\geq 2(n-1)$ for all simple pairs except for those as in 2. and 3. 
We will count the number of copies of $B$ over $A$ in a graph $N\in \K$ when $A$ and $AB$ are $\L$-strong in $N$.
For any graph $N$ and any simple pair $(A, B)$ with $A \leq_\L N$ 
we define $\chi^N(A, B)$
to be the maximal number of pairwise disjoint graphs $B'\subset N$ such that $AB'\leq_\L N$ and $B'$
and $B$ are isomorphic over $A$.  Note that in the cases where $\mu(A,B)=1$, we see that  $\chi^N(A,B)>1$ implies $N\notin\K$.
Let now $\K_\mu$ be the subclass of $\K$ consisting of those $N\in\K$ satisfying
$\chi^N (A, B) \leq \mu(A, B)$ for every simple pair $(A, B)$ with $A\leq_\L N$. Clearly
$\K_\mu$ depends only on the values $\mu(A, B)$ where $AB$ belongs to $\K$.

The following is standard:
\begin{lemma}\label{l:alg=0} Let $N\in\K_\mu$ contain two finite subgraphs $A\leq B$, $A\leq N$.
If $\delta(B/A)=0$, then $N$ contains only finitely many copies of $B$ over $A$.
\end{lemma}

\begin{proof}
See \cite{TZ}, Lemma 10.4.6.
\end{proof}

\begin{definition}\label{d:freeamalgam}
For $A,B,C\in \K$ with $A\subseteq B,C$ we let $D=B\otimes_AC$ denote the \emph{free amalgam}
of $B$ and $C$ over $A$ obtained as the graph whose set of vertices is the disjoint union of
the vertices in $A,B\setminus A$ and $C\setminus A$ with edges given by the edges of $B$ and $C$ 
and induced edges arising from Condition 4 of Definition~\ref{d:K}.
\end{definition}

\begin{lemma}\label{l:strongamalgam}
Let  $C_0,C_1,C_2\in\K$ such that $C_0\leq_k C_1, C_0\leq C_2$ and $D=C_1\otimes_{C_0}C_2$. Then
$C_1\leq  D$ and $C_2\leq_k D$.
\end{lemma}
\begin{proof}
For any $A\subset C_1, |A|\leq k,$ we have  $\delta(A/C_0)\geq 0$. 
The only edges between $A$ and $C_2\setminus C_0$
are those induced by flags containing a line from $C_0$. Hence
we have $\delta(A/C_2)= \delta(A/C_0)\geq 0$. 
For $A\subset C_2$ of any size the same argument shows that $C_1\leq D$.
\end{proof}

\begin{remark}\label{r:addition} Note in particular that if $A\leq B,C$, then 
$\delta(B\otimes_AC)=\delta(AB)+\delta(AC)-\delta(A)$.
\end{remark}

\begin{lemma}\label{l:amalgaminK}
 Let  $C_0,C_1,C_2\in\K$ such that $ C_0\subseteq C_1,C_2$ and suppose that $C_2$ is $i$-minimal over $C_0$ and $C_0\leq_{n-1} C_1$. If $D=C_1\otimes_{C_0}C_2\not\in\K$, then $i=0$. If moreover $C_0\leq C_1$, then $C_1$ strongly contains an isomorphic copy of $C_2$ over~$C_0$.
\end{lemma}
\begin{proof}
  Suppose $D$ violates condition 1 of Definition~\ref{d:K}. Then there are two lines containing
  two points (or planes, respectively), yielding a $4$-cycle.
  This cycle must consist of $a_0,a'_0$ in $C_0$,
  $b_1\in C_1\setminus C_0$ and $b_2\in C_2\setminus C_0$ such that
  $b_1$ and $b_2$ are connected with both $a_0$ and $a'_0$, contradicting the
  assumption that $C_2$ is strong over $C_0$.
  
  Suppose $D$ violates Condition 2 of Definition~\ref{d:K}, so up to duality 
  there exists a cycle in $D$
  consisting of two points contained in a common line
  and in a common plane  (possibly induced by lines in $C_0$). Since 
  $C_0\leq_1 C_1, C_2$ it easily follows that such a cycle has to be completely contained in
  $C_1$ or $C_2$, which is impossible.
  
  If $D$ violates condition 3 of Definition~\ref{d:K}, up to duality we may assume that 
  there are two planes in $D$ containing two points which are not
  contained in a common line, yielding a cycle, some of whose edges might be induced by
  a line in $C_0$. Since $C_0\leq C_2$ and $C_0\leq_{n-1} C_1$ it easily follows
  that -- up to duality -- we have two planes $a_0,a'_0$ in $C_0$, and points
  $b_1\in C_1\setminus C_0$ and $b_2\in C_2\setminus C_0$ such that
  $b_1$ and $b_2$ are connected with both $a_0$ and $a'_0$.
  Now  minimality implies $C_2=C_0\cup\{b_2\}$. So $C_2$ is $0$-minimal
  over $C_0$ and $C_0\cup\{b_1\}$ is
  a copy of $C_2$ over $C_0$. If furthermore $C_0\leq C_1$, we have $C_0\cup\{b_1\}$ 
  strongly contained in $C_1$. 
  
  $D$ satisfies condition 4 of Definition~\ref{d:K} by definition of the free amalgam.
   
  If Condition 5 of Definition~\ref{d:K} fails in $D$, then there 
  is some $x\in D$ containing
  a $2k$-cycle $\gamma$ in $\res(x)$ with $k<n$. By Lemmas~\ref{l:delta_1} and~\ref{l:line} 
  we must have $x\in C_0$. Since $C_0, C_1$ and $C_2$ are in $\K$,  there must be proper 
  non-empty  connected pieces $\gamma_1$ and $\gamma_2$ of $\gamma$ such that
  $\gamma_1\subset C_1,\gamma_2\subset C_2$ and $\gamma_1, \gamma_2\not\subseteq C_0$.
  Furthermore we may assume that  only the endpoints of $\gamma_1$ and $\gamma_2$ 
  belong to  $C_0$.   Since $\gamma$ 
  has length $2k<2n$, at least one of $\gamma_1,\gamma_2$ has length at most $n-1$.
  On the other hand, by Remark~\ref{r:pathextension}
  the length of $\gamma_1$ and $\gamma_2$
  is at least $n-1$. Since $\gamma$ has even length $2k<2n$, it follows that 
  both $\gamma_1$ and $\gamma_2$ have length $n-1$ and $\gamma=\gamma_1\gamma_2$. 
  Let $a,b\in C_0$ be the common endpoints of $\gamma_1,\gamma_2$.
  Since $C_2$ is minimal over $C_0$, it follows that
  $C_2\setminus C_0$ consists of this path $\gamma_2\setminus \{a,b\} $ only
  with no further edges between $\gamma_2$ and $C_0$.
  Then $C_2$ is also $0$-minimal over $C_0$ 
  and $\gamma_1\subset C_1$ is isomorphic  to $\gamma_2$ over $C_0$. 
  Since $\delta(\gamma_2/C_0)=0$  it follows 
  that if $C_0\leq C_1$, then also $C_0\cup\gamma_2$ is strong in $C_1$.
  
  For Condition 6a), suppose $A\subset D$ with $|A|\geq 3, A\in\K$. If $|A\cap C_1|\geq 3$,
  then $\delta(A)\geq\delta(A\cap C_1)\geq 3(n-1)+1$ by Lemma~\ref{l:strongamalgam}. 
  If $|A\cap C_1|\leq 2$ and
  $|A\cap C_2|\geq 3$, then again $\delta(A)\geq 3(n-1)+1$. So 
  suppose $|A\cap C_i|\leq 2, i=1,2$ and so $|A|\leq 4$. 
  Since $\delta(A\cap C_i)\geq 3(n-1), i=1,2$, the claim follows
  easily from the Remark~\ref{r:addition} by considering the possible cases.
  Condition 6 b) of Definition~\ref{d:K}  follows from Lemma~\ref{l:strongamalgam}.
\end{proof}
We will frequently use the following

\begin{lemma}\label{l:closer}
 Suppose we have $A, AB\leq_\L N$ and
 there are $x\in A, y\in B$ with an $E^2$-edge between $x$ and $y$ induced by a line $z\not\in AB$. Then $\delta(B/Az)\leq\delta(B/A)-(n-2)$.
\end{lemma}
\begin{proof}

 Since $AB\leq_\L N$, the line $z$  is connected only to $x$ and to $y$ in $AB$,
 so $\delta(z/A)=n-1$ and $\delta(z/AB)= 1$.
Now $\delta(AzB)= \delta(AB)+1$ and hence \[\delta(B/Az)\leq \delta(AB)+1-(\delta(A)+n-1)=\delta(B/A)-(n-2).\]
\end{proof}

\begin{lemma}\label{l:locationofsimplepair}
Suppose we have $A,B,C\in\Kmu$ such that $B$ is $i$-minimal over $A$ for some $i\geq 1$ and $A\leq_{|B|} C$. Put $D=B\otimes_AC$. If $(A',B')$ is a simple pair with  $A'\leq_\L AB, A'\not\subseteq A, |B'|\leq |B|,$ and $A'B'\leq_\L D$, then $A'B'\subseteq AB$.
\end{lemma}
\begin{proof}
Suppose $(A',B')$ is a simple pair with $A'\leq_\L  AB,
A'B'\leq_\L D, |B'|\leq |B|$, and $B'\not\subseteq AB$.
Recall that any element in $ A'$ has an edge to some element in $B'$. Suppose $x\in A'\cap B$
is connected to $y\in B'\setminus AB$. Such an edge must be induced by a line  $z\in A$ and  so  $\delta(x/A)\leq 1$. Since $B$ is $i$-minimal over $A$ for some $i\geq 1$ it follows that $i=1$ and $B=\{x\}$. Then $B'=\{y\}\subset C\setminus A$. Since $B'$ is $0$-minimal over $A'$ we have $\delta(y/A'z)<0$ by Lemma~\ref{l:closer}. However we have $A\leq_1 C$ contradicting Corollary~\ref{c:submodular}.

Hence $y\in B'\cap AB\neq\emptyset$. Thus $\delta(B'\cap AB/A')>0$ and $\delta(B'/(A'B')\cap (AB))<0$.
Now $A'B'\leq_\L D$ and Corollary~\ref{c:submodular} imply  $\delta(B'/AB)<0$ contradicting $AB\leq_{|B|} D$ by Lemma~\ref{l:strongamalgam}.
\end{proof}

The following lemma will be crucial:
\begin{lemma}\label{l:main} Let $M$ be in $\K_\mu$ and let $A\leq AB$ be a minimal
extension with $A\subseteq M$. If $N=M\otimes_A AB\notin \K_\mu$ is witnessed by
$\chi^N (A' , B' ) >\mu (A' , B' )$ for some simple pair $(A',B')$ with $A'\leq_\L N$, there
are two possibilities for $(A', B')$:
\begin{enumerate}
\item $A'\subseteq A$ and $B$ is an isomorphic copy of $B'$ over $A$. In particular, $B$ is $0$-minimal over $A$;
\item\begin{enumerate}
    \item $B$ contains an isomorphic copy of $B'$ over $A'$ and
    \item $A'$ is contained in $A\cup B$, but not a subset of $A$.
\end{enumerate}
\end{enumerate}
\end{lemma}
\begin{proof}

We first consider the case $A'\subset M$. Since $M\in \K_\mu$ there is some copy
$B''$ of $B'$ over $A'$ with $A'B''\leq_\L N$ which intersects $B$. If $B''\not\subseteq B$, then 
 $A'\leq A'\cup (M \cap B'' ) \leq
A' \cup B''$ contradicting the minimality of $B''$ over $A'$. So $B''\subseteq B$. If $B''$ were a proper subset of $B$, minimality
of $(A, B)$ would imply that $0 < \delta(B''/A) \leq  \delta(B''/A')$ by Lemma~\ref{l:submodular},  contradicting simplicity of $B$ over $A$; hence $B''=B$.  Since
$(A', B)$ is simple, every element of $A'$ has an edge to some element of $B$. If $A'\not\subseteq A$, then
some  edge is induced by a line $z$ in $A$. However, that  would imply $\delta(B/AA')\leq \delta(B/A'z)<\delta(B/A')=0$ by Lemma~\ref{l:strongamalgam}, Corollary~\ref{c:submodular} and Lemma~\ref{l:closer}, contradicting
the assumption that $B$ is a strong extension of $A$.
Thus $A'$ must be a subset of $A$ and Case~1. holds.

Next consider the case $A'\not\subseteq M$, so $A' \cap B \neq\emptyset$. Let $x\in A'\cap B$
and suppose there
are $k> 2$ disjoint copies $B'_1,\ldots, B'_k$ of $B'$ over $A'$ contained in $M$ with $A'B_i'\leq_\L N$.

 Since $B'$ is simple over $A'$ and $x\in A'$, there are $y_i$ in each copy $B'_i, i=1,\ldots k,$ incident to~$x$. Since $B$ is strong over $A$ and $x\in B$, we have $\delta(x/A)\geq 0$. Since $k>2$, this can only happen if all edges from $x$ to the $y_i$ are induced by the same line $z\in A$.

Since $A'zB'_i\leq_\L N$ we have
$\delta(B'_j/\bigcup_{i=1,\ldots j-1}B'_iA'z)\leq \delta(B_j'/A'z)\leq -(n-2)$ by Corollary~\ref{c:submodular} and Lemma~\ref{l:closer}. By Lemma~\ref{l:closer} we also have $\delta(A'z)=\delta(A')+(n-1)$. Putting these pieces together we obtain inductively
 \[0\leq\delta(\bigcup _{i=1,\ldots k}B'_iA'z) \leq\delta (A')+(n-1)-k(n-2).\]
Hence \[k\leq\frac{\delta(A')+(n-1)}{n-2}<  \frac{2\delta(A'))}{n-2} <\frac{\delta(A')}{2}\]
since $\delta(A')\geq 2(n-1)$ and $n-2\geq 4$. So there are fewer than $\delta(A')/2$ copies of $B'$ over $A'$ contained in $M$.  

Let now  $B'_{k+1},\ldots, B'_{k+l}$ be disjoint copies of $B'$ over $A'$ intersecting both $M$ and~$B$
and $A'B'_i\leq_\L N$. 
Since the $B_i'$ are $0$-minimal over $A'$ and $B_i'A'\leq_\L N$, 
we have for each $i = k+1,\ldots k+l$:
\[\delta(B'_i/MA')\leq\delta(B'_i/(M\cap B'_i)A')<0.\]

Note that $\delta(M \cap A') \leq \delta(A' )$
since $M \cap A'$ is strong in $A'$. Inductively this yields
\[0\leq\delta(\bigcup_{i=k+1}^{k+l}B_i'\cup A'/M)\leq\delta(A'/M)-l\leq \delta(A')-l.\]

Thus, $l\leq \delta(A')$ and so there are less than $\frac{3}{2}\cdot\delta(A')$ many disjoint copies of $B'$ over $A'$ which are not contained in $B$. 
Since $\mu(A',B')\geq 2\delta(A')$ this
leaves more than $\delta(A')/2$ copies of $B'$ over $A'$ inside $B$. 

We claim that $A'\subseteq AB$. Otherwise we argue as above: any element of $A'\setminus AB$ has an induced edge
to some element of each copy of $B'$. So for any copy $B''$  of $B'$ over $A'$ inside $B$ there is a line $z$ in $A$ inducing this edge. By Lemma~\ref{l:closer} we have $\delta(B''/MA')\leq \delta(B''/A'z)\leq -(n-2)$.
Since $0\leq\delta(\bigcup B''\cup A'/M)$ and $\delta(A'/M)\leq\delta(A')$,
we see as before that there are at most $\delta(A'/M)/(n-2)<\delta(A')/2$ copies of $B'$ over $A'$ inside $B$, contradicting the assumption that $(A',B')$ witnesses $N\not\in\K_\mu$.
Thus $A'$ is contained in
$A B$, finishing the proof.
\end{proof}

\begin{corollary}\label{c:isimpleamalgam}
Suppose we have $A,B, C$ such that \begin{enumerate}
\item $A\leq AB$ is $i$-minimal for some $i\geq 1$,
\item $A\leq_{|B|} C$; and
\item $A, AB, AC\in \Kmu$.
\end{enumerate}Then $D=B\otimes_AC\in \Kmu$.
\end{corollary}
\begin{proof}
By Lemma~\ref{l:amalgaminK} we have $D\in\K$. So if $D\notin\Kmu$, the second case of Lemma~\ref{l:main}  must apply: there is a simple pair $(A',B')$
with $A'\leq_\L AB, A'\not\subseteq A$ such that $\chi^D(A',B')>\mu(A',B')$. By Lemma~\ref{l:locationofsimplepair} we have $\chi^D(A',B')=\chi^{AB}(A',B')$, contradicting the assumption that $AB\in\Kmu$.
\end{proof}

\begin{definition} We say that $M \in\K_\mu$ is $\K_\mu$-saturated if for all finite $A \leq M$
and strong extensions $C$ of $A$  with $C\in\K_\mu$ there is a strong embedding of
$C$ into $M$ fixing $A$ elementwise. 
\end{definition}

Since the empty graph belongs to $\K_\mu$ and is
strongly embedded in every $A\in\K_\mu$, this implies that a $\K_\mu$-saturated $M$ strongly embeds every finite $A \in \K_\mu$.

\begin{theorem}\label{t:amalgam}
The class $\K^{fin}_\mu$ of finite elements of $K_\mu$ has the joint embedding and the amalgamation property with
respect to strong embeddings. Hence there exists a countable $K_\mu$-saturated structure $M_\mu$, which is unique up to isomorphism and is a geometry of type $\bullet \overset{n}{-} \bullet \overset{n}{-}\bullet$.
\end{theorem}

\begin{proof} Since the empty graph is in $\Kmu$ and strong in
  $A\in\Kmu$, it suffices to prove the amalgamation property. Let
  $C_0,C_1,C_2\in\K^{fin}_\mu$ with $C_0\leq C_1,C_2$. We have to find some
  $D\in\Kmu$ which contains $C_1$ and $C_2$ as strong subgraphs. 
  Clearly we may assume that $C_2$ is a minimal extension of $C_0$.
  Otherwise we decompose the extension $C_2$ over $C_0$ in a series
  of minimal extensions and obtain the required amalgam by
  amalgamating these minimal extensions step by step.
  If $C_2$ is $i$-minimal for $i\geq 1$, then by Corollary~\ref{c:isimpleamalgam}
  we know that  $D=C_1\otimes_{C_0}C_2\in\Kmu$ is the required amalgam.
  
  So suppose that $C_2$ is a $0$-minimal extension of
  $C_0$.  We will show that if $D=C_1\otimes_{C_0}C_2\not\in\Kmu$, 
  then
  $C_1$ strongly contains a copy $C'_2$ of $C_2$ over $C_0$. Since $\delta(C'_2/C_0)=0$
  this then implies that $C'_2$ is strong in $C_1$. 
  By Lemma~\ref{l:amalgaminK} we may assume that $D\in\K$.
  Let $A$ be the set of elements in $C_0$ which
  are connected to a vertex in $B=C_2\setminus C_0$ by a non-induced edge. Then $(A,B)$ is a
  simple pair and we have $D=C_1\otimes_AAB$. Since $D\notin\Kmu$ we have $\grad{D}{(A',B')}>\mu(A',B')$ for a
  simple pair $(A',B')$.
  We now apply
  Lemma~\ref{l:main}.  By assumption the second case of the lemma 
  is excluded by Lemma~\ref{l:locationofsimplepair} 
  and the assumption that $C_2\in\Kmu$. So
  the first case applies and we have $A'=A$ and $B'$ is a copy of $B$
  over $A$. All other copies $B''$ of $B$ over $A$ are contained in
  $C_1$ by simplicity. Since $B''$ is minimal over $A$, either $B''$
  must be a subset of $C_0$ or a subset of $C_1\setminus C_0$. Since
  $C_2$ is in $\Kmu$, there is a $B''$ contained in $C_1\setminus
  C_0$. Then $C_0\cup B''$ is isomorphic to
  $C_2=C_0\otimes_AAB$ over $C_0$.

  Thus
  $\K^{fin}_\mu$ has the joint embedding and the amalgamation property with
  respect to strong embeddings and we obtain a countable
  $\Kmu$-saturated structure $M_\mu$, which is unique up to
  isomorphism. In particular, in $M_\mu$ any partial isomorphism
  $f\colon A\to A'$ with $A,A'\leq M_\mu$ extends to an automorphism
  of $M_\mu$. It now
  follows as in \cite{Te1} that $M_\mu$ is a geometry of the
  required type.  
\end{proof}

\section{Model theoretic properties of $\Mmu$}

We now turn to the model theoretic properties of $\Mmu$ and
axiomatize its theory. For this we note the following:

\begin{lemma}\label{l:closure}
Let $A,B, M\in\K_\mu$.  If $A,B\leq M$, then $A\cap B\leq M$. 
\end{lemma}
\begin{proof}
Clearly, if $A,B\leq M$, then  $A\cap B\leq_\L M$.
Hence for any $C\subset M$ 
we have $\delta(C/A\cap B)\geq\delta(C/A)\geq 0$ by Corollary~\ref{c:submodular}.
\end{proof}

\begin{definition}\label{d:cl}
For $A,M\in K$, $A$ finite, we let 
\[ \cl^M(A)=\bigcap\{B\subset M\colon A\subset B\leq M\}.\]
 \end{definition}
We omit the superscript whenever it is clear from the context.
Note that for a finite set $A$ its closure 
$\cl^M(A)$ remains finite and is strong in $M$.

\begin{proposition}\label{p:strongext}
 For every $i$-minimal extension $A\leq AB$ with $i\geq 1$ and $A\leq_{|B|+n-1} \Mmu$ there exists
a copy $B'$ of $B$ over $A$ such that
$\cl(A)B'$ is  strongly embedded into $\Mmu$ and isomorphic to $ \cl(A)\otimes_AB'$.
In particular,  if $A\leq_k\Mmu$ for $k\geq |B|+n-1$, then $AB'\leq_k\Mmu$.
\end{proposition}
\begin{proof}
Let $(A,B)$ be an $i$-minimal pair with $i\geq 1$ and $A\leq_{|B|+n-1} \Mmu$.
By Corollary~\ref{c:isimpleamalgam}  we have $D=\cl(A)\otimes_AB\in \Kmu$.
By Lemma~\ref{l:strongamalgam} we know that  $\cl(A)\leq D$ and $AB\leq_k D$.
Hence by $\Kmu$-saturation we can 
strongly embed $D=\cl(A)B$ over $\cl(A)$ into $\Mmu$.
If $A\leq_k\Mmu$ for $k\geq |B|+n-1$,
then from  $AB\leq_kD\leq\Mmu$ and  $AB\leq_\L\cl(A)B=\cl(AB)$, we conclude that $AB\leq_k\Mmu$

\end{proof}
\begin{theorem}\label{t:modelMmu}
Let $T_\mu$ be the theory (in the relational language of $3$-coloured graphs introduced above)
axiomatising the class of models $M$ such that:
\begin{enumerate}

\item $M\in\K_\mu$;
\item For any $i\geq 1$ and any $i$-minimal extension $A\leq AB$  
with  $A\leq_{|B|+n-1} M$, there is a copy $B'$ of $B$ over $A$ in $M$
such that if $A\leq_k M$ for some $k\geq |B|+n-1$, then $AB'\leq_k M$;
\item $M \otimes_A AB \notin \K_\mu$ for each simple pair $(A, B)$ with $A\leq_\L M$.
\end{enumerate}
Then $T_\mu = Th(M)$.
\end{theorem}
\begin{proof}

Note first that this forms an elementary class containing $M_\mu$: the class $\K$ is elementary and for each simple pair $(A, B)$ we can express that $\chi^M (A, B) \leq \mu(A, B)$, so 1. is first-order expressible and holds in $\Mmu$ by construction.
Property 2. is a first-order property, which holds in $M_\mu$ 
  by  Proposition~\ref{p:strongext}.
   For 3. notice that if $D  = M \otimes_A AB \notin \K_\mu$
then by Lemma~\ref{l:main}  to express the
existence of a simple pair $(A', B')$ with $\chi^D (A',B') > \mu(A', B')$ one can restrict to pairs which are contained in $A \cup B$. So this can be expressed in a first order way. To see that 3. holds for $\Mmu$, assume $D= M_\mu \otimes_A AB \in\Kmu$ for some simple pair $(A,B)$ with $A\leq_\L \Mmu$. Then for every finite $C\leq M_\mu$
  which contains $A$, the graph $C\leq C\otimes_AAB$ belongs to $\Kmu$
  and so $M_\mu$ contains a copy of $B$ over $C$. So we can construct
  in $M_\mu$ an infinite sequence of disjoint copies of $B$ over
  $A$, which is impossible. 
 
 This shows that $T_\mu\subseteq Th(M_\mu)$.
  For the reverse inclusion let $M$ be a model of $\Tmu$. We have to show 
  that $M$ is
  elementarily equivalent to $M_\mu$. Choose an $\omega$-saturated
  $M'\equiv M$. By (one direction of) the next claim $M'$ is
  $\Kmu$-saturated. Then $M'$ and
  $M_\mu$ are partially isomorphic and therefore elementarily
  equivalent.

  \noindent{\bf Claim:} $M$ is an $\omega$-saturated model of $\Tmu$
  if and only if it is $\Kmu$-saturated.

  \noindent Proof:
  Let $M\models\Tmu$ be $\omega$-saturated. To show that $M$ is
  $\Kmu$-saturated, let $A\leq M$ and $A\leq B\in\Kmu$, $B$ finite. 
  We may again assume inductively that $B$ is a minimal extension of $A$.
  If $B$ is $0$-minimal over $A$, then by Axiom 3,  $M\otimes_AB$ does not belong to $\Kmu$.
  By the proof of Theorem~\ref{t:amalgam} $M$
  contains a copy $B'$ of $B$ over $A$. Clearly, $B'$ is again strong in $M$.

  If $B$ is $i$-minimal over $A$ for some $i\geq 1$, 
  then by Axiom 2, there is a copy of $B$ 
  over $A$ strongly contained in $M$.

  Conversely, suppose $M$ is $\Kmu$-saturated. Since $M$ is partially
  isomorphic to $M_\mu$, it is a model of $\Tmu$. Choose an
  $\omega$-saturated $M'\equiv M$. Then by the above $M'$ is
  $\Kmu$-saturated. So $M'$ and $M$ are partially isomorphic, which
  implies that $M$ is $\omega$-saturated.
\end{proof} 

For further reference we state the following corollary of the proof:
\begin{corollary}\label{c:omegasaturated}
The model $M_\mu$ of the theory $T_\mu$ is $\omega$-saturated.
\end{corollary}

For finite sets $A$ we now define
\[d(A)=\min\{\delta(B)\mid A\subseteq B\subseteq M\}=\delta(cl(A)).\]
Similarly, we put $d(A/B)=d(AB)-d(B)$ if $B$ is finite
and more generally we put $d(A/B)=\min \{d(A/C)\colon C\subseteq B\mbox{ finite}\}$.

We will show that for any flag $(a,b)\in M_\mu$ the set 
\[D_{(a,b)}=\{ x\in M_\mu\colon \{a,b,x\} \mbox{ is a flag}\}\] is strongly minimal. To this end we collect some  standard lemmas whose proofs can be found in \cite{TZ}:

\begin{lemma}\label{l:Tmu_qe}
  Let $M$ and $M'$ be models of $\Tmu$. Then tuples $a\in M$ and
  $a'\in M'$ have the same type if and only if the map $a\mapsto a'$
  extends to an isomorphism of $cl(a)$ to $cl(a')$. In particular,
  $d(a)$ depends only on the type of~$a$.\qed
\end{lemma}

\begin{lemma}\label{l:Tmu_alg}
  Let $M$ be a model of $\Tmu$ and $A$ a finite subset of $M$. Then
  $a$ is algebraic over $A$ if and only if $d(a/A)=0$.
\end{lemma}
\begin{proof}
  Clearly $cl(A)$ is algebraic over $A$. If $d(a/A)=0$, there is an
  extension $B$ of $cl(A)$ with $a\in B$ and $\delta(B/cl(A))=0$. By
  Lemma~\ref{l:alg=0}, $B$ is algebraic over $\cl(A)$.

  For the converse we may assume that $M$ is $\omega$-saturated.  If
  $d(a/A)>0$, we decompose the extension $cl(A)\leq\cl(Aa)$ into a
  series of minimal extensions $cl(A)=F_0\leq\ldots\leq
  F_n=\cl(Aa)$. One extension $F_k\leq F_{k+1}$ must be $i$-minimal
  for some $i> 0$. By the proof of Theorem~\ref{t:amalgam}, $F_{k+1}$ has
  infinitely many conjugates over $F_{k}$. So $\cl(Aa)$ and therefore
  also $a$ are not algebraic over $A$.
\end{proof}

\begin{proposition}\label{P:Da_str_min}
  For any model $M$ of $\Tmu$ and any flag $(a,b)\in M$,
  the set $D_{(a,b)}$ is strongly minimal.
\end{proposition}

\begin{proof}
  Let $A$ be a strong finite subset of $M$ which contains $(a,b)$ and let
  $x$ be an element of $D_{(a,b)}$. Then $d(x/A)\leq\delta(x/A)=1$. If
  $d(x/A)=0$, then $x$ is algebraic over $A$ by the previous lemma. If
  $d(x/A)=1$, then $d(xA)=\delta(xA)$, then $Ax$ is also strong in $M$. So by Lemma~\ref{l:Tmu_qe} the
  type of $x$ over $A$ is uniquely determined. This property characterizes
  strongly minimal sets, see \cite{TZ}, Lemma 5.7.3.
\end{proof}

Note that for each $x\in \P\cup\Pi$, its residue is a generalized $n$-gon
as constructed in \cite{Te1}. These were shown to be almost strongly minimal:
\begin{theorem}\rm{(\cite{Te1} Thm 4.6)}\label{t:strminngon}
Let $x\in\P\cup\Pi$ and $x_0,x_1,x_2\in \res(x)$ such that the distance between $x_i,x_j, i\neq j, i,j=0,1,2$ in $\res(x)$ is maximal possible.
Then $\res(x)$ is contained in the definable closure of $D_{(x,x_0)}\cup\{x_1,x_2\}$.
\end{theorem}

\begin{theorem}\label{t:almstrmin}
 The theory $\Tmu$ is almost strongly minimal: there is a strongly minimal 
set $D\subseteq M_\mu$ together with a finite set $B\subseteq M_\mu$ such that  
any element of $\Mmu$ is definable over $DB$.
\end{theorem}
\begin{proof}
We start by picking a parameter set 
$B_0=(p_0, 
e_1, p_2, e_3, p_4,e_5, p_6=p_0)$, consisting of a $6$-cycle of points $p_i$ 
and planes $e_i$ in $\Mmu$ and pick $x_0\in\res(p_0)$ at maximal distance from 
$e_1, e_5$ and $x_1\in\res(e_1)$ at maximal distance from $p_0,p_2$. Note that 
$B:=B_0\cup\{x_0,x_1\}$ is a $\Kmu$-structure 
and hence can be strongly embedded into $\Mmu$, whence from now on we assume 
$B\leq \Mmu$.  By Proposition~\ref{P:Da_str_min}, the residue 
$D_{(p_0,e_1)}$ of the partial flag $(p_0,e_1)$ is a strongly minimal set. We 
will show that $\Mmu\subseteq \dcl(BD)$.\\

\noindent
{\bf Claim 1:} The residues of $p_0$ and $e_1$ are contained in 
$\dcl(BD)$.

This follows immediately from Theorem~\ref{t:strminngon}. \\

\noindent
{\bf Claim 2:} The residues of $e_3$ and $p_4$ are contained in 
$\dcl(BD)$.

It suffices to show that any point in $\res(e_3)$ 
is contained in $\dcl(BD)$, as every line is uniquely determined by any two 
points in that line. 
Thus, consider 
$p\in\res(e_3)$ arbitrary. If the points $p$ and 
$p_0$ are contained in a common line $l\in\res(p_0)\subseteq \dcl(BD)$, which 
is necessarily unique, then $p$ is the unique point 
contained in $e_3$ and $l$, as $l$ is not contained in $e_3$. 

If $p$ and $p_0$ do not intersect in a common line, then there is some 
plane $e\in \res(p_0)$ which contains the two points. Now, either $e$ and 
$e_3$ intersect 
exactly in $p$, whence $p\in\dcl(BD)$, or they intersect in some line $l$ 
which is uniquely determined by $e$ and $e_3$ and thus in $\dcl(BD)$. Now 
consider another line $l'\in\res(e_3)$ connected to $p$. If there is a plane 
in $\res(p_0)$ connected to $l'$, then $l'\in\dcl(BD)$, whence also 
$p\in\dcl(BD)$, as it is uniquely determined by $l$ and $l'$. Otherwise 
consider a new point $p'\in\res(e_3)$ connected to $l'$. If $p'$ and $ p_0$ 
intersect in a line, 
then as above, $p'\in\dcl(BD)$, whence also $p\in\dcl(BD)$, as it is uniquely 
determined by $p'$ and $l$. If $p'$ and $ p_0$ intersect in a plane $e'$, 
then 
either $e'$ and $e_3$ intersect only in $p'$ and $p'\in\dcl(BD)$, or they 
intersect in a unique 
line $l''\in\dcl(BD)$. Then $p$ lays 
on the unique path of length $4$ between $l''$ and $l$ in $\res(e_3)$, whence 
$p\in\dcl(BD)$. Hence $\res(e_3)\subseteq\dcl(BD)$, as desired. 

A symmetric 
argument shows that also 
$\res(p_4)\subseteq \dcl(BD)$.\\

\noindent
{\bf Claim 3:} If $e_7\in\res(p_0)$ and $p_8\in\res(e_1)$, then the 
residues 
of $e_7$ and $p_8$ are contained in $\dcl(BD)$.

We show the statement for $\res(e_7)$, the argument for $p_8$ is exactly the 
same. As before, it suffices to show that any point in $\res(e_7)$ is contained 
in $\dcl(BD)$. Assume $p$ to be an arbitrary point in $\res(e_7)$. Once more, 
the points $p$ and $p_4$ either intersect in a unique plane $e$ or a 
unique line $l$, contained in $\res(p_4)\subseteq \dcl(BD)$. Exactly as in 
Claim 
2, substituting $e_3$ and $p_0$ by $e_7$ and $p_4$, one can see that  
$p\in\dcl(BD)$.  \\

\noindent
{\bf Claim 4:} The residues of $p_2$ and $e_5$ are contained in 
$\dcl(BD)$. 

We show that all planes in $\res(p_2)$ are contained in $\dcl(BD)$. Let $e$ be 
an arbitrary plane in $\res(p_2)$. Then $e$ and $e_7$ intersect in a unique 
line $l$ or a unique point $p$ in $\res(e_7)\subseteq\dcl(BD)$. Exactly as 
before we show that $e$ is contained in $\dcl(BD)$. \\

\noindent
{\bf Claim 5:} Any vertex of $\Mmu$ is contained in $\dcl(BD)$.

It suffices to show that an arbitrary point $p$ is contained in $\dcl(BD)$. 
Clearly, for any point $p\in\res(e_i)$ and for any plane $e$ in $\res(p_i)$ 
for $i=0,\dots,5$ we have that $\res(p)$ and $\res(e)$ are contained in 
$\dcl(BD)$ (the proof of \textit{Claim 3} applies). Hence, if the point $p$ 
intersects with any of the $p_i$ for $i=0,2,4$ in a unique plane, it already is 
contained in $\dcl(BD)$. On the other hand, if $p$ intersects with each $p_i$ 
in 
a unique 
line $l_i$, then we obtain a substructure that contradicts the fact that $B$ is 
strongly embedded in $\Mmu$: If $l_i=l_j$ for some $i\neq j$, then the 
extension of $B$ by $l=l_i$ is an extension of negative delta. If all the $l_i$ 
are distinct, then the extension of $B$ by the $l_i$ and $p$ is an extension of 
negative delta. Hence, any point $p$ has to intersect in a unique plane with 
one of the $p_i$ and is thus definable over $BD$.
\end{proof}

\begin{corollary}\label{c:weakimaginaries}
The theory $\Tmu$ has weak elimination of imaginaries.
\end{corollary}
\begin{proof}
Since $\Mmu$ is contained in the definable closure of a strongly minimal
set $D$ and a finite set $B$, it is contained in $(D,B)^{eq}$ and the result follows from
\cite{TZ}, Lemma 8.4.11. because $\acl(B)$ is easily seen to be infinite.
\end{proof} 

In order to show that forking independence is determined by the function~$d$, we
define $d$-independence as
\[\dindep{A}{B}{C}
\mbox{ if and only if } d(A/B)=d(A/BC).\]

We will show that $d$-independence coincides with the independence coming
from $\omega$-stability.
As in \cite{Te2}  we use the following characterization of independence:

\begin{lemma}\label{l:dindep}
For $A, B, C\subset \Mmu$ we have $\dindep{A}{B}{C}$ if and only if there is $B\subseteq B'\subseteq\acl(B)$  such that
$\cl(ABC)\cong \cl(AB)\otimes_{B'} \cl(BC)$. In particular, in this case $\cl(ABC)=\cl(AB)\cl(BC)$.
\end{lemma}
\begin{proof}
Let $A,B, C\subseteq\Mmu$. We may assume that $B$ is strong in $\Mmu$, so $B=\cl(B)$.
Then
\[\delta(\cl(AB)\cl(BC))=\delta(\cl(AB))+\delta(\cl(BC))-\delta(\cl(AB)\cap\cl(BC))-e.\]
where 
\[e=(2(n-1)-1)E(\widehat{\cl(AB)},\widehat{\cl(BC)})+(n-1)(E^2(\widehat{\cl(AB)},\widehat{\cl(BC)})-F(\widehat{\cl(AB)},\widehat{\cl(BC)})\]

and $\widehat{\cl(AB)}=\cl(AB)\setminus\cl(BC)$ and $\widehat{\cl(BC)}=\cl(BC)\setminus\cl(AB)$.

Since 
\[0\leq\delta(\cl(AB)/\cl(BC))=\delta(\cl(AB)\cl(BC))-\delta(\cl(BC))\]
\[\leq \delta(\cl(AB)/\cl(AB)\cap\cl(BC))=\delta(\cl(AB))-\delta(\cl(AB)\cap\cl(BC)\] we see that $e\geq 0$.

Clearly,
\[\delta(\cl(ABC))\leq \delta(\cl(AB)\cl(BC))\]and since $B$ is closed also
\[\delta(B)\leq\delta(\cl(AB)\cap\cl(BC)).\]
Hence
\[\delta(\cl(ABC))+\delta(B)\leq\delta(\cl(AB)\cl(BC))+\delta(\cl(AB)\cap\cl(BC))+e\].\[=\delta(\cl(AB))+\delta(\cl(BC))\]
Therefore equality holds if and only if 
$e=0, \delta(B)=\delta(\cl(AB)\cap\cl(BC))$ and 
$\delta(\cl(ABC))= \delta(\cl(AB)\cl(BC))$
if and only if 
\[d(A/B)=\delta(\cl(AB))-\delta(B)=\delta(\cl(ABC))-\delta(\cl(BC))=d(A/BC).\]

>From the fact that $\cl(AB)$ and $\cl(BC)$ are strong in $\Mmu$ we see that $e=0$ if and only if  any $E^2$-edge between vertices of $\widehat{\cl(AB)}$ and $\widehat{\cl(BC)}$ is induced by a unique line in $\cl(AB)\cap\cl(BC)$ and no $E$-edges occur between the two sides.
Thus $e=0$ is equivalent to $\cl(AB)\cl(BC)=\cl(AB)\otimes_{\cl(AB)\cap\cl(BC)}\cl(BC)$ and
$\delta(B)=\delta(\cl(AB)\cap\cl(BC))$ is equivalent to $\cl(AB)\cap\cl(BC)\subseteq \acl(B)$. 
Finally note that if $\delta(\cl(ABC))= \delta(\cl(AB)\cl(BC))$, then we have $\cl(ABC)=\cl(AB)\cl(B)$, proving the lemma.
\end{proof}

For future reference we record the following corollary:
\begin{corollary}\label{c:1-ample}
Suppose $\notind{A}{B}{C}$, $\cl(AC)\cap\cl(BC)=\cl(B)$ and $\cl(ABC)=\cl(AB)\cl(BC)$.
Then there is an edge $(a,c)$ with $a\in \cl(A)\setminus\acl(B),c\in\cl(C)\setminus\acl(B)$.
\end{corollary}

We show next that forking can be described from the graph metric like in \cite{Te2}, Thm 2.35:

\begin{proposition}\label{p:nonforking}
In $\Tmu$, non-forking coincides with $d$-independence.
\end{proposition}
\begin{proof}
Since $\Tmu$ is $\omega$-stable and $\Mmu$ is $\omega$-saturated, it suffices to show that $d$-independence coincides on $\Mmu$ with non-forking.
For that we have to verify the following properties, see \cite{TZ}, Thm. 8.5.10.

 \begin{enumerate}[a)]
  \item
    \textsc{(Invariance)}\index{invariance@\textsc{Invariance}}\quad
    $\dindep{}{}{}$ is invariant under $Aut(\Mmu)$,
  \item \textsc{(Local character)}\index{localcharacter@\textsc{Local
      Character}} For any finite set $A$ and arbitrary set $B$ there 
      is a finite set $B_0\subseteq B$
      such that $\dindep{A}{B_0}{B}$.

  \item \textsc{(Existence)}\index{existence@\textsc{Existence}} 
  For any finite set $A$ and arbitrary sets $B\subseteq C$ there is some $A'\subseteq\Mmu$ with
   $A'\cong _BA$ and $\dindep{A'}{B}{C}$.
  
  \item \textsc{(Weak Boundedness)}\index{weakboundedness@\textsc{Weak
      Boundedness}} For any finite set $A\subset\Mmu$ and arbitrary sets $B\subseteq C\subseteq \Mmu$ there are only finitely many isomorphism types
      of $A'\subset \Mmu$ over $B$ with $A'\cong _BA$ and $\dindep{A'}{B}{C}$.

  \item
    \textsc{(Transitivity)}\index{transitivity@\textsc{Transitivity}}
    \quad For finite sets $A$ and arbitrary sets $B,C,D$ 
    \[\mbox{ if }\dindep{A}{B}{C} \mbox{ and } \dindep{A}{BC}{D},\mbox{ then } \dindep{A}{B}{CD}.\]
    
  \item \textsc{(Weak Monotonicity)}
    \index{monotonicity@\textsc{monotonicity}!\textsc{weak}}
    \index{weakmonotonicity@\textsc{Weak Monotonicity}}\quad
    If for a finite set $A$ and arbitrary sets $B,C$ we have  
    \[\mbox{ if }\dindep{A}{B}{C} \mbox{ and } B\subseteq D\subseteq BC,\mbox{ then } \dindep{A}{B}{D}.\]
  \end{enumerate}
It is clear that \textsc{(Invariance)} holds. Also, since $d(A/B)\leq d(A/C)$ for any 
$C\subseteq\cl(B)$ and this can decrease only finitely many times, we can find
a finite set $B_0$ such that $d(A/B)=d(A/B_0)$, so \textsc{(Local character)}
holds.
    \textsc{(Transitivity)} holds as the assumptions simply state $d(A/BCD)=d(A/BC)=d(A/BC)=d(A/B)$.
\textsc{(Existence)} follows from Lemma~\ref{l:dindep}: we may assume $B=\cl(B)$ and hence we can embed $D=\cl(AB)\otimes_B(BC)$ into $\Mmu$. This shows that for $B=\cl(B)$ there is a unique $d$-independent extension to any $C$ cotaining $B$. If $B$ is not closed, there are only finitely many ways to extend the isomorphism type of $A$ over $B$ to $A$ over $\cl(B)$. Whence
\textsc{(Weak Boundedness)} follows.
For  \textsc{(Weak Monotonicity)}  notice that
if $\cl(ABC)\cong\cl(AB)\otimes_{\cl(B)}\cl(BC)$, then for any
$D$ with $B\subseteq D\subseteq BC$ we have
 $\cl(ABD)\cong\cl(AB)\otimes_{\cl(B)}\cl(BD)$.
\end{proof}

As an immediate corollary we obtain
\begin{corollary}\label{c:MR=d}
Let $M$ be a model of $\Tmu$. For
 finite sets $A,F\subset M$ we have $\MR(F/A)\leq d(F/A)$.\footnote{In fact, one can show that equality holds.}
\end{corollary}

\section{Ampleness}

We first recall the definition of ampleness from \cite{evans}:

\begin{definition}\label{r:evans_ample}
A stable theory $T$ (weakly) eliminating imaginaries is called $n$-ample if -- possibly after naming parameters -- there are tuples $a_0,\ldots a_n$ in $M$ such that 
for $i=0,\ldots n-1$ the following holds:

\begin{enumerate}
\item $\acl(a_0,\ldots a_{i-1},a_i)\cap \acl(a_0,.\ldots a_{i-1},a_{i+1})=\acl(a_0,.\ldots a_{i-1})$;
\item $\notind{a_0\ldots a_{i-1}a_i}{}{a_{i+1}a_{i+2}\ldots a_n}$; and
\item  $\indep{a_0\ldots a_{i-1}}{a_i}{a_{i+1}a_{i+2}\ldots a_n}$.

\end{enumerate}

\end{definition}

\begin{remark}\rm{(see \cite{Te2})}\label{r:intersections}
If $a_0,\ldots a_n$  witness ampleness, it follows inductively that $\acl(a_i)\cap\acl(a_{i+1})\subseteq\acl(\emptyset)$. 
\end{remark}

\begin{theorem}\label{t:2-ample}
The theory $\Tmu$ is $2$-ample. More precisely,
for any model  $M$ of $\Tmu$ any flag $(p,l,e)$ in $M$ is a witness for the 
theory being $2$-ample.
\end{theorem}
\begin{proof}
By Proposition~\ref{p:nonforking} we easily have
\begin{enumerate}
\item $\notind{p}{}{ l}, \notind{p}{}{e}$ and $\notind{l}{}{e}$; \ and
\item $\indep{p}{l}{e}$.
\end{enumerate}
Since $\delta(A)<\delta(AB)$ for any $A\subseteq\{p,l,e\}$ and $B\neq\emptyset$
we see from Lemma~\ref{l:Tmu_alg} that every subset of $\{p,l,e\}$ is algebraically closed.
\end{proof}

Exactly as in \cite{Te2} we obtain the following characterization of $2$-ample tuples.
Note that by adding the necessary parameters to the language we may assume that ampleness is witnessed
without parameters.

\begin{lemma}\label{l:2-ample}
If $(A,B,C)$  witnesses $2$-ampleness, there exist vertices $a\in\acl(A)\setminus\acl(\emptyset), 
c\in\cl(BC))$ and a line $ b\in\acl(B)$ such 
that $(a,b,c)$ is a complete flag. 
\end{lemma}
\begin{proof}
We first show that by enlarging $A$ and $C$ we may assume $\cl(AC)=\cl(A)\cl(C)$:
if $\cl(A)\cl(C)\subsetneq \cl(AC)$, let $x\in\cl(AC)\setminus\cl(A)\cl(C)$. Then $x\in\cl(ABC)=\cl(AB)\cl(BC)$.
Hence either $x\in\cl(AB)\cap\cl(AC)\subseteq\acl(A)$ or $x\in\cl(BC)\cap\cl(AC)$. 
We may replace $A$ by $A'=\acl(A)\cap\cl(AC)$, $C$ by $C'=\cl(BC)\cap\cl(AC)$ and
$B$ by $B'=\acl(B)\cap\cl(AB)\cap\cl(BC)$.
The triple $(A',B',C')$ still witnesses $2$-ampleness
and we have $\cl(A'C')=\cl(A')\cl(C')$ and $\cl(A'B'C')=\cl(A'B')\otimes_{B'}\cl(B'C')$. 
By Corollary~\ref{c:1-ample} there is an edge $(a,c)$ between $\cl(A')$ and $\cl(C')$, which must be induced by a line $b$ in $B'$. Since $a,c\notin B'$ by Remark~\ref{r:intersections},
this yields the required line $b\in B'$ yielding the flag. 
\end{proof}

Again as in \cite{Te2} this now easily implies:

\begin{theorem}\label{t:not3-ample}
The theory $\Tmu$ is not $3$-ample. 
\end{theorem}
\begin{proof}
Assume that $A,B,C,D$ witness $3$-ampleness.
Then $(A,C,D)$ witnesses $2$-ampleness and by
Lemma~\ref{l:2-ample} there exists a complete flag $(a,c,d)$ with $a\in\acl(A)\setminus\acl(\emptyset),c\in\acl(C)$ and $d\in\cl(CD)$. The edge $(a,c)$ is not induced by a line in $\acl(B)$ and
hence we must have $a$ or $c$ in $\cl(B)$. Since by Remark~\ref{r:intersections} both $\cl(A)\cap\cl(B)$ and $\cl(B)\cap\cl(C)$ are
contained in $\acl(\emptyset)$, this yields a contradiction.
\end{proof}
\begin{small}

\end{small}

\vspace{.5cm}
\noindent\parbox[t]{15em}{
Katrin Tent\footnote{Corresponding author}, Isabel M\"uller\\
Mathematisches Institut,\\
Universit\"at M\"unster,\\
Einsteinstrasse 62,\\
D-48149 M\"unster,\\
Germany,\\
{\tt tent@uni-muenster.de, i.mueller.berlin@gmail.com }}

\end{document}